\documentclass{article}
\usepackage{setspace}
\usepackage{amssymb,amsmath,amsthm,longtable}
\usepackage{subfigure}
\usepackage{calc}
\usepackage{verbatim}
\usepackage{enumerate}
\usepackage{cite}
\usepackage{epsfig}
\usepackage{graphicx}
\usepackage{graphics}
\usepackage {tikz}
\usetikzlibrary{automata,arrows,calc,positioning}
\usetikzlibrary{decorations.pathreplacing}
\usepackage{xcolor}\usepackage[hmargin=2cm,vmargin=2.5cm]{geometry}
\usepackage{chngcntr}
\usepackage[]{algorithm2e}

\newtheorem{definition}{Definition}[section]
\newtheorem{lemma}[definition]{Lemma}

\newtheorem{theorem}[definition]{Theorem}
\newtheorem{corollary}[definition]{Corollary}
\newtheorem{observation}[definition]{Observation}

\newtheorem{conjecture}[definition]{Conjecture}

\newcommand{\R}{\mathcal{R}}
\newcommand{\M}{\mathcal{M}}

\DeclareUnicodeCharacter{202F}{FIX ME!!!!}
\DeclareUnicodeCharacter{2212}{FIX ME!!!!}

\begin{document}
\title{Reconfiguration Graphs for Minimal Domination Sets}

\author{Iain Beaton\thanks{Corresponding author}\\
\small Department of Mathematics \& Statistics\\[-0.8ex]
\small Acadia University\\[-0.8ex] 
\small Wolfville, CA\\
\small\tt iain.beaton@acadiau.ca\\
}

\date{November 1, 2024}

\maketitle              

\begin{abstract}
A dominating set $S$ in a graph is a subset of vertices such that every vertex is either in $S$ or adjacent to a vertex in $S$. A minimal dominating set $M$ is a dominating set such that $M-v$ is not a dominating set for all $v \in M$. In this paper we introduce a reconfiguration graph $\R(G)$ for minimal dominating sets under a generalization of the token sliding model. We give some preliminary results which include showing that $\R(G)$ is connected for trees and split graphs. Additionally we classify all graphs which have $\R(G) = K_n$ and  $\R(G) = \overline{K_n}$ for all $n$.

\end{abstract}

\section{Introduction}\label{sec:intro}
The reconfiguration problem asks whether we can find a step-by-step transformation between two feasible solutions of a problem. Interest regarding combinatorial reconfiguration has steadily increased during the last 15 years. One well studied problem is the problem regarding the reconfiguration of dominating sets \cite{haddadan2016complexity, lokshtanov2018reconfiguration, mouawad2017parameterized, suzuki2016reconfiguration, bonamy2021dominating, kvrivstan2023shortest, rautenbach2021reconfiguring, adaricheva2021reconfiguration}. A subset of vertices $S$ of a (finite, undirected) graph $G=(V,E)$ is a {\em dominating set} if and only if every vertex of $G$ is either in $S$ or adjacent to a vertex of $S$. In reconfiguration problems, dominating sets are often represented with tokens, where exactly one token is placed on each vertex in the dominating set. Then, a reconfiguration corresponds to shifting the token according to some rule. Reconfiguration of dominating sets has been mainly studied using one of three rules:

\begin{itemize}
\item[•] \underline{Token Addition and Removal:} one can add or remove a token;
\item[•] \underline{Token Jumping:} one can move a token to any vertex of the graph;
\item[•] \underline{Token Sliding:} one can slide a token along an edge.
\end{itemize}

In this paper we will investigate the reconfiguration of minimal dominating sets. A dominating set is considered \emph{minimal} if it does not contain a smaller dominating set as a proper subset. Equivalently, for a dominating set $M$, we say $M$ is minimal if $M-v$ is not a dominating set for all $v \in M$. Reconfiguration of minimum dominating sets has been studied \cite{edwards2018reconfiguring, finbow2019gamma, lemanska2020reconfiguring}. Although, dominating sets of minimum size are necessarily also minimal, there is not much known about reconfiguring a minimal dominating set to another minimal dominating set with a different size. The first challenge is that none of the three models mentioned above will suffice for this kind of reconfiguration. Token Jumping and Token Sliding do not alter the size of the dominating set. Token Addition and Removal can increase or decrease the size by one. However, the addition of a vertex to a minimal dominating set will by definition no longer be minimal. Moreover, the removal of a vertex to a minimal dominating set will by definition no longer be a dominating set. Additionally, many graphs have ``gaps" in their sizes of minimal dominating sets. The most extreme example is the graph $K_{1,n}$. It has exactly two minimal dominating sets , one of size $1$ and the other of size $n$, which are both shown in Figure \ref{fig:StarMinimal}. 

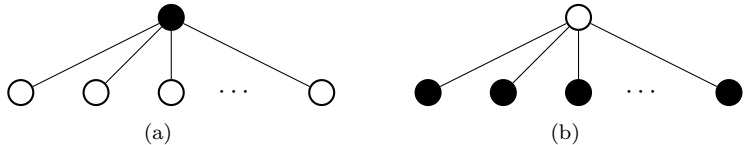
\begin{figure}[!h]
\def\c{1}
\centering
\subfigure[]{
\scalebox{\c}{
\begin{tikzpicture}
\begin{scope}[every node/.style={circle,thick,draw}]
    \node[shape=circle,draw=black,fill=black] (1) at (2,0) {};
    \node[shape=circle,draw=black,fill=white] (2) at (0,-1) {};
    \node[shape=circle,draw=black,fill=white] (3) at (1,-1) {};
    \node[shape=circle,draw=black,fill=white] (4) at (2,-1) {};
    \node[shape=circle,draw=black,fill=white] (5) at (4,-1) {};

\end{scope}

\begin{scope}
    \path [-] (1) edge node {} (2);
    \path [-] (1) edge node {} (3);
    \path [-] (1) edge node {} (4);
    \path [-] (1) edge node {} (5);

\end{scope}
\node[text width=0.75cm] at (3,-1) {$\cdots$};
\end{tikzpicture}}}
\qquad
\subfigure[]{
\scalebox{\c}{
\begin{tikzpicture}
\begin{scope}[every node/.style={circle,thick,draw}]
    \node[shape=circle,draw=black,fill=white] (1) at (2,0) {};
    \node[shape=circle,draw=black,fill=black] (2) at (0,-1) {};
    \node[shape=circle,draw=black,fill=black] (3) at (1,-1) {};
    \node[shape=circle,draw=black,fill=black] (4) at (2,-1) {};
    \node[shape=circle,draw=black,fill=black] (5) at (4,-1) {};

\end{scope}

\begin{scope}
    \path [-] (1) edge node {} (2);
    \path [-] (1) edge node {} (3);
    \path [-] (1) edge node {} (4);
    \path [-] (1) edge node {} (5);

\end{scope}

\node[text width=0.75cm] at (3,-1) {$\cdots$};
\end{tikzpicture}}}

\caption{The two minimal dominating sets of $K_{1,n}$}%
\label{fig:StarMinimal}%
\end{figure}

\noindent The ``rule(s)" required to reconfigure a minimal dominating set become more complicated than altering a single token. In this paper we propose a generalization of the token sliding model.

\begin{definition}[Expansion/Contraction Model]
\label{def:ReconfigRule}
A minimal dominating set $M_1$ may be reconfigured to another dominating set $M_2$ if there exists a vertex $v$ such that one of the following holds:

\begin{itemize}
\item (Expand) $M_2-M_1=\{v\}$ and $M_1-M_2 \subseteq N(v)$, or 
\item (Contract) $M_1-M_2=\{v\}$ and $M_2-M_1 \subseteq N(v)$.
\end{itemize}

\end{definition}

Note that when $|M_2-M_1| = |M_1-M_2|=1$ this is exactly the token sliding model. Under the expansion/contraction model, the two minimal dominating sets of $K_{1,n}$ shown in Figure \ref{fig:StarMinimal} are reconfigurations of each other. Intuitively, each reconfiguration can be thought of as a \emph{expansion} or \emph{contraction} of a vertex with respect to the minimal dominating set. For a minimal dominating set $M$ which contains vertex $v$, we say we can \emph{expand} $v$ if we can remove $v$ and add some of its neighbours to form a new minimal dominating set. Similarly, if a minimal dominating set $M$ does not contain $v$, we say we can \emph{contract} $v$ if we add $v$ and remove some of its neighbours to form a new minimal dominating set.

For reconfiguration problems it is common to study a corresponding auxiliary graph. The $k$-dominating graph (See \cite{haas2014k, haas2017reconfiguring, mynhardt2019connected, adaricheva2021reconfiguration, rautenbach2021reconfiguring}) has the dominating sets of order atmost $K$ as its vertex set where two dominating sets are adjacent if they can be reconfigured using under the token addition and removal model. The $\gamma$-graph of a graph (See \cite{finbow2019gamma, edwards2018reconfiguring}) has the minimum dominating sets as its vertex set where two minimum dominating sets are adjacent if they can be reconfigured using under the token sliding model. In this paper we will study the reconfiguration graph $\R(G)$. The vertex set of $\R(G)$ is the set of all minimal dominating sets in $G$ which we will denote $\M(G)$. Moreover two minimal dominating sets $M_1$ and $M_2$ are adjacent in $\R(G)$, denoted $M_1 \sim M_2$, if they satisfy the reconfiguration rule from Definition \ref{def:ReconfigRule}. As our model is a generalization of the token sliding model we get the following observation.

\begin{observation}
\label{obs:gamma}
For a graph $G$, $\R(G)$ contains the $\gamma$-graph of $G$ as an induced subgraph. In particular, if all minimal dominating sets are the same size then $\R(G)$ is exactly the $\gamma$-graph of $G$.
\end{observation}

\noindent A few notable examples of graphs where $\R(G)$ is simply the $\gamma$-graph of $G$ are given below

\begin{theorem}
\label{thm:families1}
Let $n$ be a positive integer.

\begin{itemize}
\item[$(i)$] $\R(\overline{K_n})=K_1$.
\item[$(ii)$] $\R(K_n)=K_n$.
\item[$(iii)$] $\R(C_5)=C_5$.
\end{itemize}

\end{theorem}

This paper is structured as follows

\section{Constructions of $\R(G)$}
\label{sec:properites}

\noindent In this section we will give some constructions of $\R(G)$ for various families of graphs and graph operations. We begin with disjoint union of two graphs. If $G$ and $H$ are graph, let $G \cup H$ denote their disjoint union. Unsurprisingly, like the token-sliding model, we will show $\R(G \cup H)$ becomes the \emph{Cartesian product} of the reconfiguration graphs of $G$ and $H$. In general, the Cartesian product of two graphs $G$ and $H$, denoted $G \square H$, has vertex set $V(G) \times V(H)$. Two vertices $(u,v)$ and $(u',v')$ are adjacent in $G \square H$ when either

\begin{itemize}
\item $u=u'$ and $v \sim_G v'$, or
\item  $v=v'$ and $u \sim_H u'$.
\end{itemize}

\begin{theorem}
\label{thm:disjointunion}
Let $G$ and $H$ be graphs. Then $\R(G \cup H) = \R(G) \square \R(H)$.
\end{theorem}

\begin{proof}
Note $M$ is a minimal dominating set of $G \cup H$ if and only if $M_G$ and $M_H$ are minimal dominating sets in $G$ and $H$ respectively.
Let $M=(M_G, M_H)$ where $M_G \in \M(G)$ and $M_H = \M(H)$.
Moreover, note that we can only expand or contract a vertex in either $G$ or $H$.
Thus $(M_G, M_H) \sim (M_G', M_H')$ in $\R(G \cup H)$ if and only if either $M_G=M_G'$ and $M_H \sim M_H'$ in $\R(H)$, or $M_H=M_H'$ and $M_G \sim M_G'$ in $\R(G)$.
\end{proof}

A special case of Theorem \ref{thm:disjointunion} is when one of the graphs is the empty graph $\overline{K_n}$. Recall that $\R(\overline{K_n})=K_1$. Moreover, the Cartesian product of any graph $G$ with $K_1$ is just $G$. This gives a useful, and intuitive, corollary. 

\begin{corollary}
Let $G$ and $n$ a positive integer. Then $\R(G \cup \overline{K_n}) = \R(G)$.
\end{corollary}

For joins of two graphs we must deal with the case where $\gamma(G)=1$ and $\gamma(G) \geq 2$ (i.e. graphs with or without universal vertices) separately. We first consider graphs with universal vertices. 

\begin{theorem}
\label{thm:joinK1}
Let $G$ be a graph. Then $\R(G \vee K_1) =  \R(G) \vee K_1$.
\end{theorem}

\begin{proof}
Let $x$ be the one vertex in $K_1$.
Every minimal dominating set in $G$ is still minimal in $K_1 \vee G$.
Thus the minimal dominating sets of $K_1 \vee G$ are exactly the sets in $\M(G)$ and $\{x\}$.
Any minimal dominating set in $\M(G)$ can be contracted to $\{x\}$.
Thus $\{x\}$ is a universal vertex in $\R(K_1 \vee G)$.
Moreover two minimal dominating sets in $\M(G)$ are adjacent in $\R(K_1 \vee G)$ if and only if they were adjacent in $\R(G)$.
\end{proof}

If a graph $G$ has a universal vertex $v$, then $v$ will also be universal if $G$ is joined to any graph $H$. Therefore it suffices to consider the joins of two graphs without universal vertices. 

\begin{theorem}
\label{thm:joinGeneral}
Let $G$ and $H$ be graphs each with no universal vertices. Then $\R(G \vee H)$ is the graph formed by joining $\R(G) \cup \R(H)$ and $G \square H$ where $M \in \M(G) \cup \M(H)$ is adjacent to $(u,v) \in V(G \square H)$ if and only if $M \cap \{u,v\} \neq \emptyset$.
\end{theorem}

\begin{proof}
The minimal dominating sets of $G \vee H$ are exactly the sets in $\M(G)$, $\M(H)$, and $V(G) \times V(H)$.
As neither $G$ nor $H$ have universal vertices then $\gamma(G), \gamma(H) \geq 2$.
Thus every set in $\M(G)$ and $\M(H)$ has at least two elements.
Therefore in $\R(G \vee H)$, no set of $\M(G)$ is adjacent to any set of $\M(H)$.

Now consider $(u,v) \in V(G) \times V(H)$ and $M \in \M(G) \cup \M(H)$.
Without loss of generality let $M \in \M(G)$.
Note that $v \notin M$ and $|M| \geq 2$.
Thus $M \cap \{u,v\} \neq \emptyset$ if and only if $u \in M$.
If $u \in M$ then $v$ can  be expanded to $M-\{u\}$.
Thus $\{u,v\} \sim M$ in $\R(G \vee H)$.
If $u \notin M$ then $|M-\{u,v\}|\geq 2$ and $|\{u,v\}-M| \geq 2$.
Therefore $\{u,v\} \not\sim M$ in $\R(G \vee H)$.
Therefore $M \in \M(G) \cup \M(H)$ is adjacent to $(u,v) \in V(G \square H)$ if and only if $M \cap \{u,v\} \neq \emptyset$.

Lastly we consider the subgraphs of $\R(G \vee H)$ induced by $\M(G)$, $\M(H)$, and $V(G) \times V(H)$.
Clearly $\M(G)$ and $\M(H)$ will induce $\R(G)$ and $\R(H)$ respectively in $\R(G \vee H)$.
Now let $(u,v),(u',v') \in V(G) \times V(H)$.
Note that $(u,v) \sim (u',v')$ in $\R(G \vee H)$ if and only if either $u=u'$ and $v \sim v'$ in $H$, or $v=v'$ and $u \sim u'$ in $G$.
Therefore $V(G) \times V(H)$ induces $G \square H$ in $\R(G \vee H)$.
\end{proof}

If either $G$ or $H$ has a universal vertex, we can construct $\R(G \vee H)$ as follows. Suppose there are $r$ universal vertices between both $G$ and $H$. First remove all universal vertices from $G$ and $H$ to form $G'$ and $H'$ respectively. Then use Theorem \ref{thm:joinGeneral} to obtain $\R(G' \vee H')$. Then add the $r$ universal vertices by iterating Theorem \ref{thm:joinK1} to get $\R(G \vee H) = \R(G' \vee H') \vee K_r$.

We now consider the reconfiguration graphs of certain families of graphs. We begin with complete bipartite graphs. 

\begin{theorem}
\label{thm:Kmn}
Let $m$ and $n$ be positive integers. Then

\begin{itemize}
\item[$(i)$] $\R(K_{1,n})=K_2$.
\item[$(ii)$] If $m,n \geq 2$ then $\R(K_{m,n})=K_{2,mn}$.
\end{itemize}
\end{theorem}

\begin{proof}
$(i)$ Note that $K_{1,n} \cong K_1 \vee \overline{K_n}$.
Moreover $\R(\overline{K_n})=K_1$.
Therefore from Theorem \ref{thm:joinK1} we have $\R(K_{1,n})=K_2$.

\noindent $(ii)$ Again note that $K_{m,n} \cong \overline{K_m} \vee \overline{K_n}$.
Additionally, $\R(\overline{K_m})=K_1$ and $\R(\overline{K_n})=K_1$ with the only minimal dominating sets being $V(\overline{K_m})$ and $V(\overline{K_n})$ respectively.
Note that every $(u,v) \in V(\overline{K_m} \square \overline{K_n})$ intersects with both $V(\overline{K_m})$ and $V(\overline{K_n})$.
Moreover $\overline{K_m} \square \overline{K_n}= \overline{K_{mn}}$.
Therefore by Theorem \ref{thm:joinGeneral} we have that $\R(K_{m,n})=(K_1 \cup K_1) \vee \overline{K_{mn}} \cong K_{2,mn}$.
\end{proof}

We will now consider complete multipartite graphs. First we must introduce a family of graphs which we will show are the reconfiguration graphs of complete multipartite graphs. The \emph{$n$ by $n$ rook's graph} is the graph $K_n \square K_n$ which represents the moves a rook can make from each square of an $n$ by $n$ chessboard. Equivalently if we let the set $\{(i,j):1 \leq i,j \leq n\}$ denote the vertices of the rook's graph then $(i,j) \sim (k,\ell)$ if and only if $i=k$ or $j=\ell$. We call \emph{$n$ by $n$ folded rook's graph} the graph which represents the moves a rook can make from each square of chessboard which has been folded along its diagonal. Let the set $\{(i,j):1 \leq j \leq i \leq n\}$ denote the vertices of the folded rook's graph. Then $(i,j) \sim (k,\ell)$ if and only if $\{i,j\} \cap \{k, \ell \} \neq \emptyset$. This is equivalent to taking the rook's graph and sequentially contracting the vertex $(i,j)$ to the vertex $(j,i)$. Examples of 3 by 3 rook's graphs and folded rook's graphs are given in Figure \ref{fig:rookgraph}

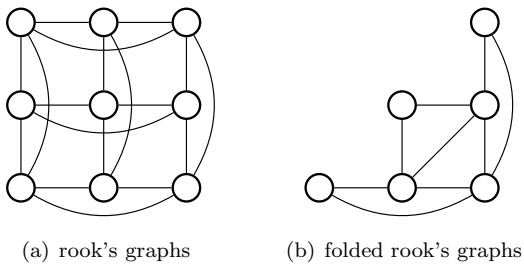
\begin{figure}[!h]
\def\c{1.1}
\centering
\subfigure[rook's graphs]{
\scalebox{\c}{
\begin{tikzpicture}
\begin{scope}[every node/.style={circle,thick,draw}]
    \node[shape=circle,draw=black,fill=white] (00) at (0,0) {};
    \node[shape=circle,draw=black,fill=white] (01) at (0,1) {};
    \node[shape=circle,draw=black,fill=white] (02) at (0,2) {};
    \node[shape=circle,draw=black,fill=white] (10) at (1,0) {};
    \node[shape=circle,draw=black,fill=white] (11) at (1,1) {};
    \node[shape=circle,draw=black,fill=white] (12) at (1,2) {};
    \node[shape=circle,draw=black,fill=white] (20) at (2,0) {};
    \node[shape=circle,draw=black,fill=white] (21) at (2,1) {};
    \node[shape=circle,draw=black,fill=white] (22) at (2,2) {};

\end{scope}

\begin{scope}
    \path [-] (00) edge node {} (01);
    \path [-] (01) edge node {} (02);
    \path[bend right=30] [-] (00) edge node {} (02);
    \path [-] (10) edge node {} (11);
    \path [-] (11) edge node {} (12);
    \path[bend right=30] [-] (10) edge node {} (12);
    \path [-] (20) edge node {} (21);
    \path [-] (21) edge node {} (22);
    \path[bend right=30] [-] (20) edge node {} (22);

    \path [-] (00) edge node {} (10);
    \path [-] (10) edge node {} (20);
    \path[bend right=30] [-] (00) edge node {} (20);
    \path [-] (01) edge node {} (11);
    \path [-] (11) edge node {} (21);
    \path[bend right=30] [-] (01) edge node {} (21);
    \path [-] (02) edge node {} (12);
    \path [-] (12) edge node {} (22);
    \path[bend right=30] [-] (02) edge node {} (22);
\end{scope}

\end{tikzpicture}}}
\qquad
\subfigure[folded rook's graphs]{
\scalebox{\c}{
\begin{tikzpicture}
\begin{scope}[every node/.style={circle,thick,draw}]

    \node[shape=circle,draw=black,fill=white] (00) at (0,0) {};
    \node[shape=circle,draw=black,fill=white] (11) at (1,1) {};
    \node[shape=circle,draw=black,fill=white] (10) at (1,0) {};
    \node[shape=circle,draw=black,fill=white] (22) at (2,2) {};
    \node[shape=circle,draw=black,fill=white] (21) at (2,1) {};
    \node[shape=circle,draw=black,fill=white] (20) at (2,0) {};

\end{scope}

\begin{scope}
    \path [-] (11) edge node {} (10);
    \path [-] (22) edge node {} (21);
    \path [-] (21) edge node {} (20);
    \path[bend left=30] [-] (22) edge node {} (20);

    \path [-] (11) edge node {} (21);
    \path [-] (00) edge node {} (10);
    \path [-] (10) edge node {} (20);
    \path[bend right=30] [-] (00) edge node {} (20);
    
    \path [-] (21) edge node {} (10);
\end{scope}
\end{tikzpicture}}}

\caption{3 by 3 rook's graph and folded rook's graphs}%
\label{fig:rookgraph}%
\end{figure}

Let  $n_1, \ldots, n_\ell$ be positive integers. We define the \emph{altered folded rook's graph} $AFR(n_1, \ldots, n_\ell)$ in the following way.
Let $n=n_1+ \ldots+ n_\ell$ and $N=\{1,2,\ldots n\}$.
Moreover let $N_k$ be the set of $n_k$ integers from $n_1+ \cdots n_k+1$ to $n_1+ \cdots n_{k+1}$.
$AFR(n_1, \ldots, n_\ell)$ is constructed from an $n$ by $n$ folded rook's graph with the following alterations:

\begin{itemize}
\item For each $k$, the vertices in $N_k \times N_k$ are contracted to a single vertex. Relabel this vertex to be $k$.
\item For each $k \neq m$, the edges between any two vertices in $N_k \times N_m$ are removed.
\end{itemize}

\noindent Note that $AFR(m,n) \cong K_{2,mn}$. Figure \ref{fig:alteredfoldedrookgraph} shows $(a)$ a 4 by 4 folded rook's graph and $(b)$ the altered folded rook's graph $AFR(1,1,2)$.

\begin{figure}[!h]
\def\c{1}
\centering
\subfigure[4 by 4 folded rook's graphs]{
\scalebox{\c}{
\begin{tikzpicture}
\begin{scope}[every node/.style={circle,thick,draw}]
    
    \node[shape=circle,draw=black,fill=white] (00) at (0,0) {(1,1)};
    \node[shape=circle,draw=black,fill=white] (10) at (2,0) {(2,1)};
    \node[shape=circle,draw=black,fill=white] (20) at (4,0) {(3,1)};
    \node[shape=circle,draw=black,fill=white] (30) at (6,0) {(4,1)};
    
    \node[shape=circle,draw=black,fill=white] (11) at (2,2) {(2,2)};
    \node[shape=circle,draw=black,fill=white] (21) at (4,2) {(3,2)};
    \node[shape=circle,draw=black,fill=white] (31) at (6,2) {(4,2)};
    
    \node[shape=circle,draw=black,fill=white] (22) at (4,4) {(3,3)};
    \node[shape=circle,draw=black,fill=white] (32) at (6,4) {(4,3)};

    \node[shape=circle,draw=black,fill=white] (33) at (6,6) {(4,4)};

\end{scope}

\begin{scope}
    \path [-] (00) edge node {} (10);
    \path[bend right=30] [-] (00) edge node {} (20);
    \path[bend right=30] [-] (00) edge node {} (30);
    \path [-] (10) edge node {} (20);
    \path[bend right=30] [-] (10) edge node {} (30);
    \path [-] (20) edge node {} (30);    
    \path [-] (11) edge node {} (21);
    \path[bend right=30] [-] (11) edge node {} (31);
    \path [-] (21) edge node {} (31);       
    \path [-] (22) edge node {} (32); 
    
    \path [-] (33) edge node {} (32);
    \path[bend left=30] [-] (33) edge node {} (31);
    \path[bend left=30] [-] (33) edge node {} (30);
    \path [-] (32) edge node {} (31);
    \path[bend left=30] [-] (32) edge node {} (30);
    \path [-] (31) edge node {} (30);   
    \path [-] (22) edge node {} (21);
    \path[bend left=30] [-] (22) edge node {} (20);
    \path [-] (21) edge node {} (20);        
    \path [-] (11) edge node {} (10); 
    
    \path [-] (10) edge node {} (21);
    \path [-] (10) edge node {} (31);
    
    \path [-] (32) edge node {} (21);
    \path [-] (32) edge node {} (20);

\end{scope}

\end{tikzpicture}}}
\qquad
\subfigure[$AFR(1,1,2)$]{
\scalebox{\c}{
\begin{tikzpicture}
\begin{scope}[every node/.style={circle,thick,draw}]
    
    \node[shape=circle,draw=black,fill=white] (00) at (0,0) {1};
    \node[shape=circle,draw=black,fill=white] (10) at (2,0) {(2,1)};
    \node[shape=circle,draw=black,fill=white] (20) at (4,0) {(3,1)};
    \node[shape=circle,draw=black,fill=white] (30) at (6,0) {(4,1)};
    
    \node[shape=circle,draw=black,fill=white] (11) at (2,2) {2};
    \node[shape=circle,draw=black,fill=white] (21) at (4,2) {(3,2)};
    \node[shape=circle,draw=black,fill=white] (31) at (6,2) {(4,2)};
    

    \node[shape=circle,draw=black,fill=white] (33) at (5,5) {3};

\end{scope}

\begin{scope}
    \path [-] (00) edge node {} (10);
    \path[bend right=30] [-] (00) edge node {} (20);
    \path[bend right=30] [-] (00) edge node {} (30);
    \path [-] (10) edge node {} (20);
    \path[bend right=30] [-] (10) edge node {} (30);
 
    \path [-] (11) edge node {} (21);
    \path[bend right=30] [-] (11) edge node {} (31);

    \path [-] (31) edge node {} (30);   
    \path [-] (21) edge node {} (20);        
    \path [-] (11) edge node {} (10); 
    
    \path [-] (10) edge node {} (21);
    \path [-] (10) edge node {} (31);

    \path [-] (33) edge node {} (21);
    \path[bend left=10] [-] (33) edge node {} (20);
    \path [-] (33) edge node {} (31);
    \path[bend right=10] [-] (33) edge node {} (30);
\end{scope}
\end{tikzpicture}}}

\caption{Examples of folded and altered folded rook's graph}
\label{fig:alteredfoldedrookgraph}%
\end{figure}
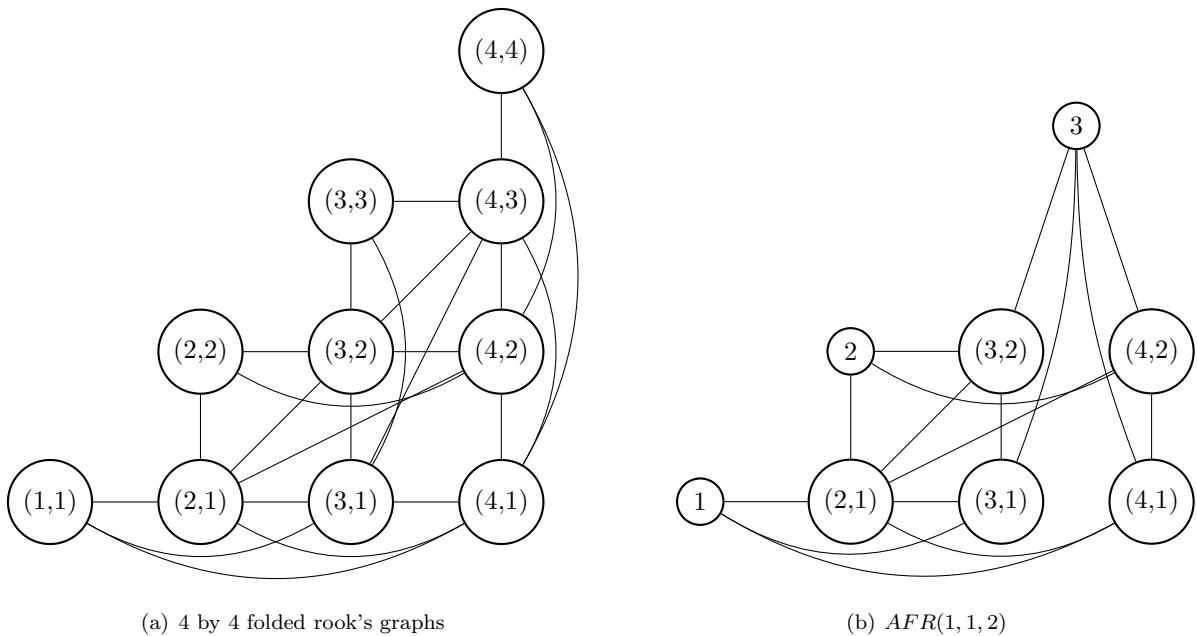

For complete multipartite graphs $K_{n_1, \ldots, n_\ell}$, we need only to consider cases where each $n_i \geq 2$. Otherwise if some $n_i=1$ then we would have a universal vertex and hence Theorem \ref{thm:joinK1} would apply.

\begin{theorem}
\label{thm:MultiParitite}
Let $\ell, n_1, \ldots n_\ell \geq 2$ be integers. Then $\R(K_{n_1, \ldots, n_\ell})=AFR(n_1, \ldots, n_\ell)$.
\end{theorem}

\begin{proof}
For simplicity let $G$ denote $K_{n_1, \ldots, n_\ell}$.
Let $N_1, \ldots, N_\ell$ denote the vertices in each part of $G$.
The minimal dominating sets of $G$ are exactly the sets $N_1, \ldots, N_\ell$, and the pairs $\{i,j\}$ where $i \in N_k$ and $j \in N_m$ with $m \neq k$.
Let $n=n_1+ \ldots+ n_\ell$.
Now relabel the vertices of $G$ from 1 to $n$ such that such that each $N_k$ contains the $n_k$ integers from $n_1+ \cdots n_k+1$ to $n_1+ \cdots n_{k+1}$. We claim that there is an isomorphism $f$ from $\R(G)$ to $AFR(n_1, \ldots, n_\ell)$ such that $f(N_k)=k$ and $f(\{i,j\})=(i,j)$ when $i>j$.

In $G$, each $N_k$ can be contracted to $\{i,j\}$ if and only if $N_k \cap \{i,j\} \neq \emptyset$.
Moreover, as each $n_k, n_m \geq 2$ then $N_m \not\sim N_k$ in $\R(G)$ when $k \neq m$.
Note that $f(N_k)=k$.
In $AFR(n_1, \ldots, n_\ell)$, the vertex $k$ is the contraction $N_k \times N_k$ in the the folded rook's graph.
Thus $k$ is adjacent to $(i,j)$ if and only if $\{i,j\} \cap N_k \neq \emptyset$.
Moreover $k \neq m$ are not adjacent in $AFR(n_1, \ldots, n_\ell)$.

In $G$, consider the minimal dominating set $\{i,j\}$ where $i \in N_k$, $j \in N_m$, and $k \neq m$.
Without loss of generality suppose $i>j$.
Then $\{i,j\}$ has the following neighbours in $\R(G)$:

\begin{itemize}
\item all $\{i',j\}$ with $i' \notin N_k$,
\item all $\{i,j'\}$ with $j' \notin N_m$,
\item $N_k$ and $N_m$.
\end{itemize}

\noindent Note that $f(\{i,j\})=(i,j)$.
In $AFR(n_1, \ldots, n_\ell)$, the vertex $(i,j)$ is adjacent to $k$ and $m$. 
Additionally, the only other neighbours of $(i,j)$ are $(i',j')$ where $\{i,j\} \cap \{i',j'\} \neq \emptyset$ and $(i',j') \notin N_k \times N_m$.
Thus  $(i,j)$ has the following neighbours in $AFR(n_1, \ldots, n_\ell)$:

\begin{itemize}
\item all $(i',j)$ and $(j,i')$ with $i' \notin N_k$,
\item all $(i,j')$ and $(j',i)$ with $j' \notin N_m$,
\item $k$ and $m$.
\end{itemize}

\noindent Therefore $f$ is an isomorphism from $\R(G)$ to $AFR(n_1, \ldots, n_\ell)$.

\end{proof}

We will now consider the reconfiguration of the $n$ by $n$ rooks graph $K_n \square K_n$. First we impose vertex labelling of the $n$-fold cartesian product $K_n \square \cdots \square K_n$ where every vertex is labelled by the $n$-tuple $(v_1, \ldots, v_n)$ where each $1 \leq v_i \leq n$. We call  $(v_1, \ldots, v_n)$ a \emph{permutation} vertex if its $n$-tuple is a permutation of $(1,2, \ldots, n)$. That is,  $(v_1, \ldots, v_n)$ a \emph{permutation} vertex if $\{v_1, \ldots, v_n\}=\{1,2,\ldots, n\}$.

\begin{theorem}
\label{thm:RooksGraph}
For any positive integer $n$, $\R(K_n \square K_n)$ is the graph formed from two copies of the $n$-fold Cartesian product $K_n \square \cdots \square K_n$ such that every permutation vertex is contracted to its respective copy.
\end{theorem}

\begin{proof}
Let $v_{ij}$ denote the vertex in the $i^{th}$ row and $j^{th}$ column of $K_n \square K_n$.
Any dominating set of $K_n \square K_n$ must either have a vertex in every row or every column.
This is because if a set of vertices omitted some row $i$ and some column $j$, then $v_{ij}$ would not be dominated.
Moreover, every set which has a vertex in every row or every column is a dominating set.
Thus the minimal dominating sets are the sets which contain exactly one vertex in every row or exactly one vertex in every column.
Without loss of generality, consider the minimal dominating sets which have exactly one vertex in every row of $K_n \square K_n$.
Let $\M_r$ denote the set of all such minimal dominating sets.
Label each minimal dominating set $(c_1, \dots, c_n) \in \M_r$ where $c_i$ is column index of the vertex in row $i$.
Note that in the reconfiguration graph $\R(K_n \square K_n)$,  $(c_1, \dots, c_n) \sim (c_1', \dots, c_n')$ if and only if $(c_1, \dots, c_n)$ and $(c_1', \dots, c_n')$ differ by exactly one index.
Therefore $\M_r$ induces the $n$-fold Cartesian product $K_n \square \cdots \square K_n$ in $\R(K_n \square K_n)$.
Let $\M_c$ denote the the minimal dominating sets which have exactly one vertex in every column of $K_n \square K_n$.
By the generality of our argument, $\M_c$ also induces the $n$-fold Cartesian product $K_n \square \cdots \square K_n$ in $\R(K_n \square K_n)$.
Note that $\M_r \cap \M_c$ are the sets which have exactly one vertex in every row and exactly one vertex in every column.
Thus $\M_r \cap \M_c$ is collection of all the permutations of $(1,2, \ldots, n)$.

It now suffices to show that if $M_r \in \M_r-\M_c$ and $M_c \in \M_c-\M_r$ then $M_c \not\sim M_r$ in $\R(K_n \square K_n)$.
To show a contradiction suppose that $M_c \sim M_r$ in $\R(K_n \square K_n)$.
Note that $|M_c|=|M_r|$ thus $|M_c-M_r|=1$ and $|M_r-M_c|=1$.
As $M_r \in \M_r-\M_c$ then $M_r$ has two vertices which share the same column of $K_n \square K_n$.
As $M_c$ has exactly one vertex in each column of $K_n \square K_n$ then $M_r-M_c=\{x\}$ where $x$ is one of these two vertices.
Moreover, $M_c-M_r=\{y\}$ where $y$ is in a different column than $x$ and thus in the same row as $x$.
Note $M_c=M_r\cup\{y\}-\{x\}$.
However, $y$ is in the same row as $x$ and $M_r \in \M_r$ and hence had exactly one vertex in every row.
Thus $M_c$ also has exactly one vertex in every row.
Therefore $M_c \in \M_r$ which contradicts $M_c \in \M_c-\M_r$.
\end{proof}

A graph $G$ is a \emph{threshold graph} if $G \equiv K_1$ or it is formed by adding an isolated or universal vertex to a threshold graph. Alternatively, a threshold graph is formed with a $K_1$ and then sequentially adding an isolated or universal vertices. The vertices in a threshold graph can be ordered $v_1, \ldots, v_n$ in such a way that each $v_i$ is adjacent to either none or all vertices $v_j$ with $1\leq j<i$. If $v_i$ is adjacent to all vertices which come before it, then it was added as an universal vertex. Moreover if $v_i$ is adjacent to no vertices which come before it, then it was added as an isolated vertex. The complete graph $K_n$ and empty graph $\overline{K_n}$ are both examples of threshold graphs. The complete graph $K_n$ had $n-1$ vertices added as universal vertices whereas $\overline{K_n}$ had no vertices added as universal vertices. 

In the next theorem we show that $\R(G)$ for threshold graphs are complete graphs. In fact we show that only threshold graphs have their reconfiguration graph as a complete graph. Before we do so we need the following well-known observation.

\begin{observation}
\label{obs:complement}
Let $G$ be a graph with no isolated vertices. If $M$ is a minimal dominating set of $G$ then $V-M$ is a dominating set.
\end{observation}

\begin{theorem}
\label{thm:threshold}
Let $G$ be a graph and $r$ a positive integer. $\R(G) = K_r$ if and only if $G$ is a threshold graph where $r-1$ vertices were added as universal vertices.
\end{theorem}

\begin{proof}
We will first prove the converse statement through induction on $r$.
Suppose $r=1$, then $G$ is the empty graph and hence $\R(G) = K_1$.
Now suppose the statement holds for some $r \geq 1$.
Let $G$ be a threshold graph where $r$ vertices were added as universal vertices.
Note that if $G$ has any isolated vertices every minimal dominating set must contain them, so $\R(G)$ is not affected by their removal.
So, without loss of generality assume that $G$ has no isolated vertices.
Then $G \cong (H \vee K_1)$ where $H$ is a threshold graph with $r-1$ vertices were added as universal vertices.
By our induction hypothesis it follows that $\R(H)=K_r$.
Therefore by Theorem \ref{thm:joinK1}, $\R(G)=K_1 \vee K_r = K_{r+1}$.

Now suppose for some graph $G$ that $\R(G) = K_r$.
Without loss of generality assume that $G$ has no isolated vertices.
Again we proceed by induction on $r$.
For $r=1$ suppose that $\R(G) = K_1$.
Then $G$ has exactly one minimal dominating set.
Note that $G$ has no edges. 
Otherwise, as $G$ has no isolated vertices it would follow from Observation \ref{obs:complement} that $G$ has at least two minimal dominating sets. 
Therefore $G$ has no edges and hence must be an empty graph.
Indeed, $G$ is a threshold graph where $0$ vertices were added as universal vertices.

Now suppose for some $r \geq 1$ that $\R(G) = K_r$ implied that $G$ was a threshold graph where $r-1$ vertices were added as universal vertices.
We will first show that $G$ must have at least one universal vertex.
To show a contradiction, suppose $G$ has no universal vertices.
Then every minimal dominating set of $G$ has at least two vertices.
Let $M$ be a minimal dominating set $M$
Note $V-M$ is a dominating a dominating set, so let $M^{-1}$ be any minimal dominating set within $V-M$.
Note $M^{-1}$ and $M$ are disjoint.
Therefore $|M^{-1}-M|=|M^{-1}| \geq 2$ and $|M-M^{-1}|=|M| \geq 2$.
Therefore no one vertex can be contracted to expanded to reconfigure $M$ to $M^{-1}$.
Therefore $M \not\sim M^{-1}$ in $\R(G)$ which contradicts $\R(G) = K_r$.
Therefore $G$ has a universal vertex.

Let $G \cong H \vee K_1$.
It follows from Theorem \ref{thm:joinK1} that $\R(G) \cong \R(H) \vee K_1$
Moreover $\R(G) = K_r$ so it must be that $\R(H)=K_{r-1}$.
By our inductive hypothesis this implies that $H$ is a threshold graph where $r-1$ vertices were added as universal vertices.
Therefore $G$ is a threshold graph where $r$ vertices were added as universal vertices.
\end{proof}

We now give a useful lemma regarding independent sets. For a graph $G$, an \emph{independent set} is a subset of vertices such that no two vertices in $S$ are adjacent in $G$.

\begin{lemma}
\label{lem:subgraph}
Let $G$ be a graph $G$ with independent set $S$. Let $\mathcal{M}_S$ be the collection of minimal dominating sets in $G$ which every vertex in $S$ but no vertex in $N(S)$. Then the minimal dominating sets of $\mathcal{M}_S$ induce $\R(G-N[S])$ as a subgraph of $\R(G)$.
\end{lemma}

\begin{proof}
We first show that $M_S \in \mathcal{M}_S$ if and only if $M_S=M \cup S$ for some minimal dominating set $M$ of in $G-N[S]$.
Suppose $M_S \in \mathcal{M}_S$ and let $M=M_S-S$.
Note that $M \subseteq V(G-N[S])$.
Moreover, no vertex in $S$ is adjacent to any vertex in $G-N[S]$.
Therefore $M$ must be a dominating set of $G-N[S]$.
To show a contradiction, suppose that $M$ is not a minimal domination set of $G-N[S]$. 
Then there exists a proper subset $M' \subset M$ which dominates $G-N[S]$.
Note that the vertices in $S$ and $M'$ dominate $N[S]$ and $G-N[S]$ respectively.
Therefore $M' \cup S$ is a dominating set in $G$ and is a proper subset of $M_S$.
However, this contradicts the fact that $M_S$ is a minimal dominating set in $G$.
Therefore $M_S=M \cup S$ for some minimal dominating set $M$ of in $G-N[S]$.

Let $M$ be a minimal dominating set in $G-N[S]$.
We will show that $M \cup S$ is a minimal dominating set in $G$.
The vertices in $S$ and $M$ dominate $N[S]$ and $G-N[S]$ respectively.
Therefore $M \cup S$ is a dominating set in $G$.
Note that no vertex in $M$ is adjacent to any vertex in $S$.
Now for any $v \in M$, as $M$ is minimal, then there exists a vertex in $G-N[S]$ which is not dominated by $M-\{v\}$.
Thus $M \cup S-\{v\}$ is not a dominating set for any $v \in M$.
For any $v \in S$, as $S$ is an independent set no vertex in $S-\{v\}$ is adjacent to $v$.
Thus $M \cup S-\{v\}$ is not a dominating set for any $v \in S$.
Therefore $M \cup S$ is a minimal dominating set for any minimal dominating set $M$ in $G-N[S]$.

Now let $M_1 \cup S$ and $M_2 \cup S$ be two minimal dominating sets in $\mathcal{M}_S$.
Note $M_1$ and $M_2$ must be minimal dominating sets of $G-N[S]$.
Note that $M_1 \cup S-M_2 \cup S=M_1-M_2$ and similarly $M_2 \cup S-M_1 \cup S=M_2-M_1$.
Therefore $M_1 \sim M_2$ in $\R(G-N[S])$ if and only if $M_1 \cup S \sim M_2 \cup S$ in $\R(G)$.
Therefore $\mathcal{M}_S$ induces $\R(G-N[S])$ as a subgraph of $\R(G)$.
\end{proof}

We now show that only the empty graph $\overline{K_r}$ has $\R(G)$ also as an empty graph.
Lemma \ref{lem:subgraph} lends itself nicely to proof by induction.
If know that $\R(G)$ is empty, it must be true that $\R(G-N[v])$ is also empty.
Induction will take care of every case except the one where $\R(G-N[v])$ is empty for every vertex $v$.
This leads us to our next Theorem.
We remark that classifying the graphs by the graphs induced by $G-N[v]$ has been a recent area of interest (See \cite{YU2021112519, zhang2022graphs, zhang2024graphstree,zhang2024graphsreg }).

\begin{theorem}
\label{thm:G-N[v]_Empty}
 $G$ is a graph with the property that for all vertices $v$, there exists an $r$ such that $G-N[v] \cong \overline{K_r}$ if and only if $G \cong \overline{K_n}$ or $G$ is a complete multipartite graph.
\end{theorem}

\begin{proof}
First suppose $G$ is a complete multipartite graph.
For any vertex $v$ note that $v$ is adjacent to all vertices except those in the same part as $v$.
Therefore $G-N[v] \cong \overline{K_{r}}$ where $r+1$ is number of vertices in the same part as $v$.
Moreover if $G \cong \overline{K_n}$, then $G-N[v] \cong \overline{K_{n-1}}$ for all vertices $v$.

Now suppose $G$ is a graph with the property that for all vertices $v$, there exists an $r$ such that $G-N[v] \cong \overline{K_r}$.
We will now iteratively construct $G$ by looking at the graphs induced by the open neighbourhoods of its vertices.
Fix a vertex $a_1 \in V(G)$ and let $A_1=V(G)-N(a_1)$
As $G-N[a_1] \cong \overline{K_r}$ for some $r$ and $a_1$ is not adjacent to any other vertex in $A_1$ then $A_1$ induces an independent set.
We now claim that each vertex $A_1$ is adjacent to every vertex in $N(a_1)$.
Suppose some $x \in A_1$ is not adjacent to some $y \in N(a_1)$.
Then $V-N[x]$ would induce the edge $a_1y$ which contradicts that $G-N[x]$ is empty.
Therefore $G \cong \overline{K_{|A_1|}} \vee H_1$ where $H_1$ is the graph induced by $N(a_1)$.
Consider two cases for $H$.

\vspace{2mm}
\noindent \underline{Case 1}: Suppose $H_1$ has no edges.
Then $G \cong \overline{K_{|A_1|}} \vee \overline{K_{|H_1|}}$.
If $|H_1|=0$ then $G$ is empty and if $|H_1|>0$ then $G$ is a complete bipartite graph.
\vspace{2mm}

\vspace{2mm}
\noindent \underline{Case 2}: Suppose $H_1$ has at least one edge.
Note that every $v \in V(H_1)$ is adjacent to every vertex in $A_1$.
Therefore $G-N_G[v] \cong H_1-N_{H_1}[v]$ for every $v \in V(H_1)$.
Thus $H$ has the property that for all vertices $v \in V(H_1)$, there exists an $r$ such that $H_1-N_{H_1}[v] \cong \overline{K_r}$.

\vspace{2mm}

We can iterate this process by choosing $a_2 \in V(H_1)$ and let $A_2=V(H_1)-N_{H_1}(a_2)$.
The same argument would show that $H_1 \cong \overline{K_{|A_2|}} \vee H_2$ where $H_2$ is the graph induced by $N_{H_1}(a_2)$.
If $H_2$ has no edges then the process terminates and $G$ is either a complete bipartite graph with parts $A_1$ and $A_2$ or a complete triparite graph with parts $A_1$, $A_2$, and $V(H_2)$.
If $H_2$ has at east one edge then the process continues.
Note that the process will eventually terminate as the order of $H_i$ decreases as the process continues.
Let $k$ be the number of iterations until the process terminates. 
Then $G$ is a complete multipartite graphs with $k$ or $k+1$ parts.
\end{proof}

\begin{theorem}
\label{thm:emptygraph}
Let $G$ be a graph and $r$ a positive integer. $\R(G) = \overline{K_r}$ if and only if $G$ is an empty graph and $r=1$.
\end{theorem}

\begin{proof}
Suppose $G$ is an empty graph $\overline{K_n}$ and $r=1$.
Then as $\overline{K_n}$ only has one minimal dominating set then $\R(G) = \overline{K_r}$.

We will now show that every non-empty graph $G$ also has a non-empty $\R(G)$.
We will induct on the number of vertices $n$ in $G$.
The smallest non-empty graph is on two vertices. 
Additionally, there is exactly one non-empty graph is on two vertices $K_2$, and $\R(K_2)=K_2$.
Now suppose every non-empty graph up to order $n-1$ has a non-empty reconfiguration graph.
Let $G$ be a non-empty graph of order $n$.
Let $v$ be a vertex in $G$.
By our inductive hypothesis if $G-N[v]$ is non-empty then $\R(G-N[v])$ is non-empty.
Moreover, from Lemma \ref{lem:subgraph} it follows that $\R(G-N[v])$ is a induced subgraph of $\R(G)$ and hence $\R(G)$ is non-empty.
So if any vertex $v \in V(G)$ has $G-N[v]$ is non-empty then $\R(G)$ is non-empty.
So suppose for every vertex $v \in V(G)$ that $G-N[v]$ is empty.
As $G$ is non-empty then it follows from Theorem \ref{thm:G-N[v]_Empty} that $G$ is a complete multiplartite graph.
If any part of $G$ is size one, then $G$ has a universal vertex $x$.
It follows from Theorem \ref{thm:joinK1} that $\R(G) = \R(G-x) \vee K_1$ and hence $\R(G)$ has an edge.
If every part of $G$ is size at least 2 then it follows form Theorem \ref{thm:MultiParitite} that $\R(G)$ is an altered folded rook graph and hence has an edge.
Therefore $\R(G)$ is non-empty.
\end{proof}

\section{When is $\R(G)$ connected}
\label{sec:conneect}

In this section we will investigate when the reconfiguration graph is connected. However, as we will show for a general graph $G$, its reconfiguration graph $\R(G)$ is necessarily not connected. To give some intuition for when a reconfiguration is possible, we will state some useful definitions from \cite{2021Beaton} which categorize vertices according to a dominating set $S$. For a graph $G$, let $\mathcal{D}(G)$ denotes the collection of dominating sets of $G$. For a dominating set $S$ of $G$ let

$$a(S) = \{v \in S : S-v \notin \mathcal{D}(G)\},$$

\noindent denote the set of \emph{critical} vertices of $S$ with respect to domination (in that their removal makes the set no longer dominating). Note that a dominating set $S$ is minimal if and only if $S=a(S)$. To contrast critical vertices, we say a vertex in $S$ is \emph{supported} if it is not critical. That is, $v$ is a supported vertex of $S$ if $v \in S-a(S)$.

We say a supported vertex $v \in S - a(S)$ is \emph{supported by} $u \in S$ if $N[u] \cap N[v] \neq \emptyset$. This brings us to our first observation regarding critical vertices.

\begin{observation}
\label{obs:IsoImpliesCrit}
 Let $G$ be a graph with dominating set $S$ containing $v$. If $v$ has no neighbours in $S$ then $v \in a(S)$.
\end{observation}

\noindent We will partition the vertices not in $S$ into the following two sets:

\vspace{-6mm}

\begin{align*}
N_1(S) &= \{v \in V-S : |N[v] \cap S| = 1 \}\\
N_2(S) &= \{v \in V-S : |N[v] \cap S| \geq 2 \}.
\end{align*}

\noindent Furthermore, recall the partition $a(S) = a_1(S) \cup a_2(S)$, where 

\vspace{-6mm}

\begin{align*}
a_1(S) &= \{v \in a(S) : N[v] \cap N_1(S) \neq \emptyset \}\\
a_2(S) &= \{v \in a(S) : N[v] \cap N_1(S) = \emptyset \}.
\end{align*}

For a dominating set $S$, if $v$ is critical then some vertex is not dominated by $S-\{v\}$. We will call the vertices not dominated by $S-\{v\}$ the \emph{private neighbours} of $v$ (with respect to $S$). Whether $v \in a_1(S)$ or $v \in a_2(S)$ will help us locate its private neighbours. If $v \in a_1(S)$ then it has a private neighbour in $N_1(S)$. If $v \in a_2(S)$ then $v$ is the only private neighbour of $v$.

\begin{observation}
\label{obs:a2ImpliesIso}
 Let $G$ be a graph. For any dominating set $S$ if $v \in a_2(S)$ then $N(v) \subseteq N_2(S)$ and hence has no neighbours in $S$.
\end{observation}

\begin{proof}
Let $v \in a_2(S)$.
As $v$ is critical then $S-v$ is not a dominating set and hence has a private neighbour.
As $S$ was a dominating set then any private neighbour of $v$ is in $N[v]$.
More specifically, either $v$ or some vertex in $N(v)-S$ is not dominated by $S-v$.
By the definition of $a_2(S)$, $N[v] \cap N_1(S) = \emptyset$.
Therefore every vertex in $N(v)-S$ must be in $N_2(S)$, and thus is still dominated by its other neighbour in $S$.
Thus, $v$ is not dominated by $S-v$ and hence it has no neighbours in $S-v$.
Therefore  $N(v) \subseteq N_2(S)$.
\end{proof}

We will now show that the reconfiguration graph of a forest is connected. We will often refer to a vertex of degree $1$ as a \emph{leaf}. Additionally, if a vertex has a leaf neighbour, we refer to it as a \emph{stem}. The following lemma will be useful.

\begin{lemma}
\label{lem:treeRefig}
Let $T$ be a tree with minimal dominating set $M$. Let $s$ be a stem in $T$ and let $L$ be the set of leaf neighbours of $s$. If $M$ contains $s$ with no non-leaf neighbours in $N_1(M)$ then $M'=M \cup L -\{s\}$ is a minimal dominating set and adjacent to $M$ in $\R(T)$.
\end{lemma}

\begin{proof}
First note that $M'$ is formed by expanding $s \in M$.
Thus it suffices to show that $M'$ is a minimal dominating set.
We now show that $M'$ is a dominating set.
Consider the private neighbours of $s$ with respect to $M$.
Any vertex not dominated by $M-\{s\}$ must be in $N[s]$.
No neighbour of $s$ which was a non-leaf was in $N_1(M)$.
Thus each non-leaf neighbour of $s$ is either in $M-\{s\}$ or was in $N_2(M)$.
In both cases the non-leaf neighbours of $s$ are still dominated by $M-\{s\}$.
Thus any private neighbour of $s$ is a leaf in $L$ or $s$ itself.
Thus $M-\{s\}$ dominates every vertex in $T$ except possibly $s$ and the leaves in $L$.
As $L$ dominates both $s$ and $L$ then $M'=M \cup L -\{s\}$ is a dominating set of $T$.

To show $M'$ is minimal, we will show that every vertex in $M'$ is critical.
Each vertex in $L$ is not adjacent to any other vertex in $M'$.
Thus by Observation \ref{obs:IsoImpliesCrit}, each vertex in $L$ is critical.
Now consider any vertex $v \in M-\{s\}$.
As $M$ was a minimal dominating set of $T$ and $v \in M$, then $v$ has a private neighbour $u$ with respect to $M$.
Note that $u \notin N[s]$ otherwise it would also have $s$ as a neighbour in $M$.
Furthermore $u \notin N[L]$ as $N[L] \subseteq N[s]$.
Therefore $u$ is still a private neighbour of $v$ in $M'$.
Thus $v$ is critical in $M'$ and $M'$ is a minimal dominating set of $T$.
\end{proof}

\begin{theorem}
If $G$ is a forest, then $\R(G)$ is connected.
\end{theorem}

\begin{proof}
We proceed by induction on the number of edges $m$ of $G$.
If $m=0$ then $G \cong \overline{K_n}$ and hence $\R(G)=K_1$ which is connected.

Now suppose every forest with at most $m \geq 0$ edges has a connected reconfiguration graph.
Let $G$ be a forest with $m+1$ edges.
Note if every tree in $G$ has at most $m$ edges then each component of $G$ has a connected reconfiguration graph.
Moreover, it follows from Theorem \ref{thm:disjointunion} that $\R(G)$ is connected.
So suppose some tree $T$ in $G$ has $m+1$ edges.
Note that every other component in $G$ (should there be any) is isomorphic to $K_1$.
Therefore $\R(G)=\R(T)$.
In the case where $T$ is a star graph $K_{1,n}$ then $\R(T)=K_2$ which is connected.
So we may assume $T$ is not a star graph.
In $T$, let $s$ be a stem in $T$ with exactly one non-leaf neighbour.
We note that such a stem exists by rooting $T$ at any vertex and letting $s$ be the stem of greatest depth.
Let $L$ be the set of leaf neighbours of $s$.
Any minimal dominating set of $T$ contains $s$ or every leaf in $L$, but not both.
That is if a minimal dominating set contains $s$ then it does not contain any leaves from $L$.
Let $\mathcal{M}_s$ and $\mathcal{M}_L$ denote the minimal dominating sets of $G$ containing $s$ and $L$ respectively.
Note that $\mathcal{M}_s$ and $\mathcal{M}_L$ partition all minimal dominating sets of $G$.
Let $T_L=T-N[L]$ be the tree formed by removing the leaves in $L$ and their collective neighbour $s$ from $T$.
Note that $T_L$ has $|L|+1$ fewer edges than $T$ and hence by our inductive hypothesis $\R(T_L)$ is connected.
Since $L$ is an independent set then by Lemma \ref{lem:subgraph} we have that $\mathcal{M}_L$ induces $\R(T_L)$ as an induced subgraph of $\R(T)$.
It now suffices to show that every minimal dominating set in $\mathcal{M}_s$ is connected to some minimal dominating set in $\mathcal{M}_L$.
Let $M \in \mathcal{M}_s$ and let $v$ be the one non-leaf neighbour of $s$.
We will consider the two cases where $v \in M$ and $v \notin M$.
If $v \in M$ then let $M'=(M \cup L)-\{s\}$.
Note that $v$ is the only neighbour of $s$ in $T$ and $v \notin N_1(M)$.
Therefore by Lemma \ref{lem:treeRefig}, $M'$ is minimal and adjacent to $M$ in $\R(T)$.
Moreover $M' \in \mathcal{M}_L$, so $M$ is connected to $\mathcal{M}_L$.

Now suppose $v \notin M$ and let $\mathcal{M}_{s,v}$ denote all such minimal dominating sets of $T$.
Note that $\mathcal{M}_{s,v}$ is collection of all dominating sets of $T$ which contain $s$ but no vertex in $N(S)$.
Let $T_s=T-N[s]$ be the forest formed by removing $s$, $v$ and the leaves in $L$.
We may assume $T_s$ has at least one vertex, otherwise $T \cong K_{1,n}$ which was dealt with in an earlier argument.
Note that $T_s$ has at least $|L|+1$ fewer edges than $T$ and hence by our inductive hypothesis $\R(T_L)$ is connected.
By Lemma \ref{lem:subgraph} we have that $\mathcal{M}_{s,v}$ induces $\R(T_s)$ as an induced subgraph of $\R(T)$.
Thus is suffices to show that one minimal dominating set $M \in \mathcal{M}_{s,v}$ is connected to one minimal dominating set $M' \in \mathcal{M}_L$. 
Note that any minimal dominating set of $T_s$ union with $\{s\}$ is a minimal dominating set in $\mathcal{M}_{s,v}$
For any vertex $x$ in $T_s$, there exists a minimal dominating set of $T_s$ which contains $x$. 
Let $x$ be a vertex adjacent to $v$ in $T$ and $M_x$ be a minimal dominating set of $T_s$ containing $x$.
Now let $M=M_x \cup \{s\}$.
Then $v$ has at least two neighbours $s, x \in M$ and hence $v \in N_2(M)$.
Thus the only non-stem neighbour of $s$ in $T$ is not in $N_1(M)$.
Let $M' = M \cup L -\{s\}$.
By Lemma \ref{lem:treeRefig}, $M'$ is minimal and adjacent to $M$ in $\R(T)$.
Moreover $M' \in \mathcal{M}_L$.
Therefore $M$ and every other minimal dominating set in $\mathcal{M}_{s,v}$ is connected to a minimal dominating set in $\mathcal{M}_L$.
\end{proof}

Another family of interest is \emph{split graphs}. A graph is considered a split graph if you can partition its vertex set into a clique and an independent set. For a split graph $G$, we let $C_G$ and $I_G$ denote the clique and independent set of $G$. In \cite{BERTOSSI198437}, Bertossi showed the problem of finding a minimum cardinality dominating set is NP-complete for split graphs. Additionally, it was shown in \cite{COUTURIER2015634} that a split graph of order $n$ can have up to $3^{\frac{n}{3}}$ minimal dominating sets. Despite this will show that a split graph of order $n$ have a connected reconfiguration graph with $O(n)$ diameter. We will first give a useful lemma

\begin{lemma}
\label{lem:splitReconfig}
Let $G$ be a split graph with minimal dominating set $M$ containing some $v \in C_G$. Let $M_v$ be any minimal dominating set of the subgraph of $G$ induced by the external private neighbours of $v$ with respect to $M$. Then $M'=M \cup M_v -\{v\}$ is a minimal dominating set and hence adjacent to $M$ in $\R(T)$.
\end{lemma}

\begin{proof} 
Let $G_v$ be the subgraph of $G$ induced by the external private neighbours of $v$ with respect to $M$.
Moreover let $M'=M \cup M_v -\{v\}$ where $M_v$ is a minimal dominating set of $G_v$.
Note no vertices in $G_v$ are adjacent to any vertex in $M-\{v\}$.
Thus every vertex in $M_v$ is critical in $M'$.
Thus it suffices to show that each vertex in $M-\{v\}$ is critical in $M'$.
Let $u \in M-\{v\}$ and consider two cases: $u \in I_G$ and $u \in C_G$.
If $u \in I_G$ then $u \in a_2(M)$ as $v$ would also dominate every neighbour of $u$ in $C_G$.
Moreover no neighbour of $u$ could be in $M_v$ and hence $u \in a_2(M')$.
So suppose $u \in C_G$.
Then $u$ must have an external private neighbour $u' \in I_G$ with respect to $M$.
It suffices to show that $u'$ is not adjacent to any vertex in $M_v$.
To show a contradiction, suppose not, that is suppose some $x \in M_v$ is adjacent to $u'$.
Then $x \in C_G$ and is an external private neighbour of $v$ with respect to $M$.
However $v, u \in C_G$ and $u,v \in M$.
So $x$ was dominated by both $u$ and $v$ in $M$ which contradicts the fact that $x$ is a private neighbour of $v$ with respect to $M$.
Therefore $u$ has $u'$ as an external private neighbour in $M'$ and hence $u \in a_1(M')$.
\end{proof}

\begin{theorem}
If $G$ is a split graph, the $\R(G)$ is connected with diameter at most $2|I_G|+1$
\end{theorem}

\begin{proof}
Let $\mathcal{M}_I$ denote the minimal dominating sets of $G$ which contain every vertex in $I_G$.
Note that if $I_G$ dominates $G$ then $\mathcal{M}_I=\{I_G\}$.
Otherwise $\mathcal{M}_I=\{I_G \cup \{u\}: u \in V(G)-N[I_G]\}$.
In both cases $\mathcal{M}_I$ induces a clique in $\R(G)$.
It suffices to show that every minimal dominating set $M$ is distance at most $|I_G-M|$ from some minimal dominating set in $\mathcal{M}_I$.
We will induct on $|I_G-M|$, that is the number of vertices from $I_G$ which are in $M$.
The case where $|I_G-M|=0$ is trivial.
So suppose for some $k \geq 0$ that our claim is true for all minimal dominating sets $M$ with $|I_G-M| \leq k$.
Let $u \in I_G-M$ and consider two cases: $|M \cap N(u)|=1$ and $|M \cap N(u)| \geq 2$.
If $|M \cap N(u)|=1$ then let $v$ be the lone neighbour of $u$ in $M$.
Note that $v \in C_G$ and let $G_v$ be the subgraph of $G$ induced by the external private neighbours of $v$ with respect to $M$.
Note that $u$ is an external private neighbour of $v$ with respect to $M$.
Let $M_v$ be any minimal dominating set of $G_v$ which contains $u$.
Then by Lemma \ref{lem:splitReconfig}, $M'=M \cup M_v - \{v\}$ is minimal dominating set connect to $M$ in $\R(G)$.
Moreover $|I_G-M'| \leq k$ and thus $M'$ is distance at most $k$ from some minimal dominating set in $\mathcal{M}_I$.
Hence $M$ is distance at most $k+1$ from some minimal dominating set in $\mathcal{M}_I$. If $|M \cap N(u)| \geq 2$ then $u$ is dominated by at least two vertices in $C_G$.
As $M$ is minimal, then each of these vertices have private neighbours in $I_G$. 
Hence there exists a vertex $u' \in I_G-M$ with $|M \cap N(u')|=1$ and hence this case follows from the above argument.
\end{proof}

Although forests and split graphs have connected $\R(G)$, not all reconfiguration graphs are connected. The smallest example occurs on 6 vertices. The graph shown below, $K_3 \Box P_2$ in Figure \ref{fig:isolated} has a disconnected reconfiguration graph.

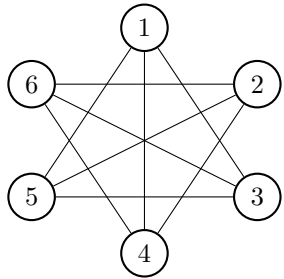
\begin{figure}[!h]
\def\c{1}
\centering
\scalebox{\c}{
\begin{tikzpicture}
\begin{scope}[every node/.style={circle,thick,draw}]
    \node[shape=circle,draw=black,fill=white] (1) at (0,1.5) {1};
    \node[shape=circle,draw=black,fill=white] (2) at (1.5,0.75) {2};
    \node[shape=circle,draw=black,fill=white] (3) at (1.5,-0.75) {3};
    \node[shape=circle,draw=black,fill=white] (4) at (0,-1.5) {4};
    \node[shape=circle,draw=black,fill=white] (5) at (-1.5,-0.75) {5};
    \node[shape=circle,draw=black,fill=white] (6) at (-1.5,0.75) {6};

\end{scope}

\begin{scope}
    \path [-] (1) edge node {} (3);
    \path [-] (1) edge node {} (4);
    \path [-] (1) edge node {} (5);
    
    \path [-] (2) edge node {} (4);
    \path [-] (2) edge node {} (5);
    \path [-] (2) edge node {} (6);
    
    \path [-] (3) edge node {} (5);
    \path [-] (3) edge node {} (6);
    
    \path [-] (4) edge node {} (6);

\end{scope}
\end{tikzpicture}}
\caption{A graph with disconnected $\R(G)$}%
\label{fig:isolated}%
\end{figure}

\noindent The reconfiguration graph $\R(K_3 \Box K_2)$ has 10 vertices. However, the vertices corresponding to the minimal dominating sets $\{1,3,5\}$ and $\{2,4,6\}$. We will now give a general construction which shows that $\R(G \Box K_2)$ is disconnected for all graphs $G$ with $\delta(G) \geq 2$.

\begin{theorem}
\label{thm:disconnected}
Let $G$ and $H$ be graphs of equal order with $\delta(G)\geq 2$ and $\delta(H)\geq 1$. Let $M(G,H)$ denote a graph from one copy of $G$ and $H$ joined by a perfect matching from $G$ to $H$. Then the reconfiguration graph of $M(G,H)$ has an isolated vertex.
\end{theorem}

\begin{proof}
Let $M$ be the set of all vertices in $G$.
Note that $M$ is a dominating set of $M(G,H)$.
Moreover, every vertex in $M$ has its matched neighbour in $H$ as a private neighbour.
Therefore $M$ is a minimal dominating set of $M(G,H)$.
Note that each vertex in $M$ has exactly one neighbour not in $M$.
Thus any expansion of $M$ would be by token slide.
Moreover each vertex not in $M$ has exactly one neighbour in $M$.
Thus any contraction of $M$ would also be by token slide.
It suffices to show that the token slide of any vertex in $M$ to its lone neighbour in $H$ would not be a minimal dominating set.
Let $v \in M$ and $v'$ be its lone neighbour in $H$.
Consider $M_v = M \cup \{v'\}-\{v\}$.
As $\delta(H) \geq 1$, then $v'$ has atleast one neighbour in $H$.
Let $u'$ be a neighbour of $v'$ in $H$.
Moreover let $u$ be the lone neighbour of $u'$ in $G$.
We will now show that $u$ is not a critical vertex in $M_v$.
Note that $u'$ is adjacent to both $u$ and $v'$ in $M_v$.
Therefore $u$ no longer has $u'$ as a private neighbour in $M_v$.
Every other neighbour of $u$ is in $G$.
However, as $\delta(G) \geq 2$ then every vertex in $G$ is dominated by atleast two vertices in $M'$.
Therefore $u$ has no private neighbours with respect to $M_v$.
Thus $M_v$ is not a minimal dominating set.
Hence the token slide of any vertex in $M$ to its lone neighbour in $H$ would not be a minimal dominating set.
\end{proof}

\begin{corollary}
\label{cor:disconnected}
Let $G$ be a graph with $\delta(G)\geq 2$, then $\R(G \Box K_2)$ is disconnected.
\end{corollary}

We remark that the conditions of Corollary \ref{cor:disconnected} can weakened slightly to include graphs $G$ with some component which is not isomorphic to $K_1$ or $K_{1,n}$. If the reconfiguration of this component is disconnected then so is its Cartesian product with the reconfiguration graph of the other components of $G$. Moreover if some component of $G$ is not isomorphic to $K_1$ or $K_{1,n}$ then every vertex $v$ is adjacent to a vertex $u$ with degree 2. It will follows from a similar argument to that of Theorem \ref{thm:disconnected} that $u$ will not be critical in $M_v$.


\section{Open Problems}
\label{sec:problems}

For a reconfiguration graphs there are several questions one might ask:

\begin{itemize}
\item[$(i)$] For what graph  $\R(G)$ connected?
\item[$(ii)$] Given a graph $H$, classify all graphs $G$ such that $\R(G)=H$.
\item[$(iii)$] Which graphs can be reconfiguration graphs?
\item[$(iv)$] When does $\R(G) \cong G$?
\item[$(v)$] What is the diameter of $\R(G)$?
\item[$(vi)$] What is the girth of $\R(G)$?
\item[$(vii)$] Can we bound the minimum and maximum degree of $\R(G)$?
\end{itemize}

For $(i)$ we showed that trees and split graphs have connected $\R(G)$.
For $(ii)$ we showed $\R(G)=K_n$ if and only if $G$ was a threshold graph.
Moreover, we showed that $\R(G)=\overline{K_n}$ if and only if $G$ was an empty graph and $n=1$.
We also offer the following conjecture.

\begin{conjecture}
Let $G$ be a graph. $\R(G)$ is a tree if and only if $G \cong \overline{K_n} \cup K_{1,m}$ for some non-negative integers $m$ and $n$.
\end{conjecture}

For $(iii)$ although not explicitly investigated in this paper, we know that not all graphs can be reconfiguration graphs.
Theorem \ref{thm:emptygraph} and Theorem \ref{thm:families1} $(i)$ imply that $\overline{K_n}$ for $n\geq 2$ can not be a reconfiguration graphs.
To that regard we suspect $\R(G)$ has girth at lost 5 with the exception of when $\R(G)$ is a tree. 

We conclude with a remark on $(vii)$. We know from Theorem \ref{thm:joinK1} that $\R(G)$ can have a universal vertex when $G$ has a universal vertex. Thus the maximum degree of $\R(G)$ for a general graph $G$ is bounded above by the number of minimal dominating sets in $G$ less one. It was shown in \cite{COUTURIER2015634} that the number of minimal dominating sets in $G$ could be exponential in terms of the order of $G$. From Theorem $\ref{thm:Kmn}$ we know that $\R(K_{n,n})$ has maximum degree $n^2$ which is still much higher than the order of $K_{n,n}$. In the next theorem we will give a family of graphs whose reconfiguration graph has maximum degree less than the order of the graph. We first will use a useful lemma comparing the size of $a_1(S)$ and $N_1(S)$.

\begin{lemma}\cite{2021Beaton}
\label{lem:a1n1}
 Let $G$ be a graph. For any dominating set $S$ of $G$ we have $|a_1(S)| \leq |N_1(S)|$.
\end{lemma} 

\begin{theorem}
\label{thm:maxdegree}
Let $G$ be a graph on $n$ vertices with girth at least 5. Then $\Delta (\R_G) \leq n-\gamma(G)$.
\end{theorem}

\begin{proof}
Let $M$ be a minimal dominating set of $G$.
It suffices to show that the degree of $M$ in $\R(G)$ is at most $n-|M|$
Every reconfiguration of $M$ either expand a vertex in $a_1(M)$ or $a_2(M)$ or contract a vertex in $N_1(M)$ or $N_2(M)$.
However, every contraction of a vertex in $N_1(M)$ is necessarily the expansion of a vertex in $a_1(M)$.
Moreover, every expansion of a vertex in $a_2(M)$ is the contraction of a vertex in $N_2(M)$.
Every expansion or contraction of $M$ involves removing at least one vertex from $M$.
Thus we need only consider expansions of $v \in a_1(M)$ and contractions of $v \in N_2(M)$.
We will show in each there there is at most one expansion or contraction of each vertex.

Let $v \in a_1(M)$ and let $P_v$ be the private neighbours of $v$ with respect to $M$.
Any expansion of $v$ takes the form $M_v=M \cup T -\{v\}$ where $T \subseteq N(v)$.
Moreover $T$ must dominate every vertex in $P_v$.
As $G$ has girth at least 5 and hence triangle free.
So no two vertices in $N(v)$ are adjacent.
Thus the only subset of $N(v)$ which can dominate $P_v$ is $P_v$ itself.
Therefore any expansion of $v$ takes the form $M \cup P_v -\{v\}$.
Thus there is at most one expansion of $M$ which expands $v$. 

Let $v \in N_2(M)$ and let $A_2=N(v)\cap a_2(M)$.
Any contraction of $v$ takes the form $M_v=M \cup \{v\} - T$ where $T \subseteq N(v) \cap M$.
We will now show that if $M_v$ is a contraction of $v$ then $T=A_2$.
Let $x \in A_2$ and to show a contradiction suppose $x \notin T$ for a contraction $M_v$.
Every $x \in A_2$ only has itself as a private neighbour with respect to $M$.
In $M_v$, both $v$ and $x$ are adjacent to $x$.
Thus for $M_v$ to be minimal, $x$ has a private neighbour $x' \notin M_v$ with respect to $M_v$.
Note $x'$ was a not a private neighbour of $x$ with respect to $M$.
So let $y \in M$ be adjacent to $x'$.
Now $x$ and $y$ are both adjacent to $x'$ and $v$.
However, this forms a 4-cycle which contradicts $G$ having girth at least 5.
Thus every $x \in A_2$ is also in $T$ so $A_2 \subseteq T$.
Now consider $x \in T-A_2$ for a contraction $M_v$.
Then $x \in M$ with $x \notin a_2(M)$ and so $x \in a_1(M)$.
Let $x' \notin M$ be a private neighbour of $x$.
As $x \notin M_v$ and $M_v$ is a dominating set then $v$ must dominate $x'$.
However, then $x$, $x'$, and $v$ form a 3-cycle which contradicts $G$ having girth at least 5.
Thus $T-A_2$ is empty and hence $T \subseteq A_2$.
Therefore $T=A_2$ for any contraction of $v$.
Thus there is at most one contraction of $M$ which contracts $v$.

Every minimal dominating set $M_v$ of $G$, with $M \sim M_v$ in $\R(G)$, is either an expansion of $v \in a_1(M)$ or a contraction of $v \in N_2(M)$.
Moreover for each $v \in a_1(M) \cup N_2(M)$ there is at most one expansion or contraction of $v$.
Thus $\deg(M) \leq |a_1(M)|+|N_2(M)|$ where $\deg(M)$ denotes the degree of $M$ in $\R(G)$.
By Lemma \ref{lem:a1n1} we have that $|a_1(M)| \leq |N_1(M)|$.
Thus $\deg(M) \leq |N_1(M)|+|N_2(M)|=n-|M|$.
\end{proof}

%
%
\bibliographystyle{abbrv}
\bibliography{mybib}

\end{document}